\documentclass[11pt]{article}

\usepackage[margin=1.12in]{geometry}
\usepackage{amsmath,amssymb,amsthm,mathtools}
\usepackage{microtype}
\usepackage{enumitem}
\usepackage{mathrsfs}
\usepackage[hidelinks]{hyperref}
\usepackage[nameinlink,capitalize]{cleveref}

\newtheorem{theorem}{Theorem}[section]
\newtheorem{proposition}[theorem]{Proposition}
\newtheorem{corollary}[theorem]{Corollary}
\newtheorem{lemma}[theorem]{Lemma}
\theoremstyle{remark}
\newtheorem{remark}[theorem]{Remark}
\newtheorem{question}[theorem]{Question}

\newcommand{\T}{\mathbb T}
\newcommand{\E}{\mathbb E}
\newcommand{\Pp}{\mathbb P}

\newcommand{\C}{\mathbb C}
\newcommand{\e}{\mathrm e}
\newcommand{\dd}{\,\mathrm d}
\newcommand{\Mellin}{\mathscr M}
\newcommand{\Mcal}{\mathcal M}

\title{\textbf{How Random Is the M\"obius Function?\\
Smoothing, Probability, and the Riemann Hypothesis}}
\author{Alberto Verjovsky}
\date{}

\begin{document}
\maketitle

\begin{abstract}
This article is primarily expository, but it also contains several new
observations and reformulations concerning the M\"obius function,
probabilistic models, dynamical systems, and the Riemann hypothesis.  Its
starting question is classical: in what sense can the M\"obius function
be said to behave randomly?  We begin with Denjoy's random-walk heuristic
and place it in the context of later work on random multiplicative
functions, short intervals, and M\"obius pseudorandomness.  We then
develop two smoothing forms of the classical Mertens criterion for the
Riemann hypothesis.  The first uses the discrete Laplace transform
\[
\Phi(t)=\sum_{n\geq1}\mu(n)e^{-nt}
\]
and identifies RH with the condition
\[
\Phi\in L^p(0,\infty)
\qquad\text{for every }1\leq p<2.
\]
The second uses normalized M\"obius Fourier polynomials and local
moments on arcs of length comparable with \(1/N\).  The paper also
revisits the author's earlier criterion for RH in terms of discrete
measures, as presented in Broughan's account of analytic equivalents of
RH.  Denjoy's heuristic is made precise in a simple independent
coefficient model, and this model is carefully distinguished from the
modern theory of random multiplicative functions.  The resulting
coefficient space also gives a measure--category contrast: the relevant
\(L^p\)-property has full measure but is topologically meagre.  Since
the M\"obius sequence itself is one distinguished deterministic point of this coefficient
space, and the property is equivalent to RH exactly at that point,
this contrast gives a precise sense in which the property required by RH is
measure-theoretically typical: it is the rigorous content behind Denjoy's
often-quoted, and necessarily heuristic, remark that the Riemann hypothesis
holds ``with probability one.''  A further
point of the paper is dynamical: the multiplicative semigroup of positive
integers acts naturally on the coefficient space, and the M\"obius
sequence is a distinguished arithmetic point for this action.  The aim
throughout is explanatory, while keeping these new observations visible:
to show how arithmetic, probability, Fourier analysis, Mellin transforms,
and dynamics fit together, and to indicate what modern results add to
the older random-walk picture.
\end{abstract}



\section{Introduction}

The M\"obius function
\[
\mu(n)=
\begin{cases}
0,&\text{if \(n\) is divisible by the square of a prime},\\
(-1)^k,&\text{if \(n\) is a product of \(k\) distinct primes},
\end{cases}
\]
is one of the simplest arithmetic functions to define and one of the
most difficult to understand globally.  Its summatory function
\[
M(x)=\sum_{n\leq x}\mu(n)
\]
measures the cumulative cancellation among its values.  A classical
form of the Mertens criterion states that
\begin{equation}\label{eq:mertens-rh}
\mathrm{RH}
\quad\Longleftrightarrow\quad
M(x)=O_\varepsilon\bigl(x^{1/2+\varepsilon}\bigr)
\qquad\text{for every }\varepsilon>0.
\end{equation}
See, for example, \cite[Ch.~XIV]{Titchmarsh},
\cite[Ch.~6]{MontgomeryVaughan}, and \cite{Tenenbaum}.

The exponent \(1/2\) suggests a probabilistic analogy.  A sum of
independent centered signs typically has size comparable with the
square root of the number of terms, and this comparison is old.
Denjoy proposed treating the nonzero values of the M\"obius function
heuristically as independent random signs, a perspective discussed by
Edwards in his account of the Riemann hypothesis
\cite[pp.~263--266]{Edwards} but not always attributed explicitly in
later accounts.  Denjoy's viewpoint long predates the modern
terminology of pseudorandomness, yet it already isolates the central
question: to what extent does the deterministic sequence
\(\mu(1),\mu(2),\ldots\) display the cancellation expected from
randomness?  Denjoy's heuristic is often summarized by the striking,
and necessarily informal, remark that the Riemann hypothesis therefore
holds ``with probability one.''  Sections~\ref{sec:probabilistic-form}
and~\ref{sec:measure-category} make this remark mathematically
precise: they show that, relative to Haar measure on the coefficient
space \(\{-1,0,1\}^{\mathbb N}\), the M\"obius sequence is one distinguished deterministic point in the
coefficient space, while the property under consideration holds on a set
of full Haar measure and is at the same time exceptional in the sense of
Baire category.  Thus ``probability one'' and ``topologically generic''
point in opposite directions.

A modern and influential formulation of the same broad philosophy is
Sarnak's M\"obius-disjointness conjecture \cite{Sarnak}.  It predicts
that the M\"obius function is asymptotically orthogonal to every
sequence generated by a zero-entropy topological dynamical system.
Sarnak's conjecture is logically distinct from RH, but both belong to
the study of M\"obius pseudorandomness: RH concerns the cancellation of
the partial sums \(M(x)\), whereas Sarnak's conjecture concerns
correlations with deterministic sequences of low dynamical complexity.
RH belongs to a much larger family of equivalent analytic
reformulations, surveyed comprehensively in Broughan's three-volume
\emph{Equivalents of the Riemann Hypothesis}
\cite{BroughanI,BroughanII,BroughanIII}; one of these equivalents
supplies a more personal starting point for the present article.  In an
earlier paper \cite{VerjovskyDiscrete} I considered a family of
discrete measures built from Euler's totient function and proved an
equivalent form of RH in terms of their rate of convergence.  Broughan
gives a detailed exposition of this result in Chapter~10 of Volume~II,
where he calls it the ``Verjovsky criterion'' \cite[Ch.~10]{BroughanII};
the proof there, as in the present paper, is governed by Mellin
transforms and by the passage between an asymptotic estimate and a
zero-free region for \(\zeta(s)\).  Returning to this older argument
was one of the reasons for looking again at RH through a smoothed
arithmetic object.

The present article approaches the Mertens criterion through two
smoothing procedures.  The first is exponential smoothing:
\begin{equation}\label{eq:Phi-def}
\Phi(t)=\sum_{n=1}^{\infty}\mu(n)e^{-nt},
\qquad t>0.
\end{equation}
Its Mellin transform is \(\Gamma(s)/\zeta(s)\).  Consequently, the
behavior of \(\Phi(t)\) near \(t=0\) encodes zero-free half-planes for
the zeta function and leads naturally to an \(L^p\)-criterion with
critical exponent \(2\).

The second procedure is Fourier-theoretic.  Define
\begin{equation}\label{eq:PN-merged}
S_N(t)=\sum_{n\leq N}\mu(n)e^{2\pi int},
\qquad
P_N(t)=\frac{S_N(t)}{\sqrt N}.
\end{equation}
For \(c>0\) and \(1\leq q<\infty\), consider the local moment
\begin{equation}\label{eq:localmoment}
\mathcal M_{q,c}(N)
=
\left(
\frac{N}{2c}
\int_{-c/N}^{c/N}|P_N(t)|^q\,dt
\right)^{1/q}.
\end{equation}
Since
\[
P_N(0)=\frac{M(N)}{\sqrt N},
\]
high moments on an arc of scale \(1/N\) detect the same square-root
cancellation that appears in \eqref{eq:mertens-rh}.

The random-sign comparison can be made exact.  Replacing the values
\(\mu(n)\) with independent variables uniformly distributed on
\(\{-1,0,1\}\), one obtains a random Laplace transform that belongs
almost surely to \(L^{2-\varepsilon}(0,\infty)\) for every
\(\varepsilon>0\).  The associated coefficient space
\(\{-1,0,1\}^{\mathbb N}\) carries a natural multiplicative semigroup
action.  Its nontrivial generators are measure preserving, exact,
mixing, and ergodic.  The same model also reveals a striking
measure--category dichotomy: the \(L^{2-\varepsilon}\) property has
full Haar measure but is topologically meagre.

This same coefficient space carries a further, dynamical structure:
the natural action
\[
(\sigma_\ell\omega)(n)=\omega(\ell n),
\qquad \ell\in\mathbb N_+.
\]
The M\"obius function determines a distinguished arithmetic point whose
orbit retains the multiplicative structure of \(\mu\).  In this
language the Laplace criterion for RH becomes a membership question for
the M\"obius point in a natural measurable subset of the coefficient
space.  This gives a connection between RH and dynamical systems that is
different from, but complementary to, Sarnak's M\"obius-disjointness
viewpoint.  It is also reminiscent of Furstenberg's dynamical approach
to number theory: arithmetic multiplication is represented by a family
of commuting transformations, and an arithmetic question is recast in
terms of the resulting dynamical system.  The analogy is one of
viewpoint rather than an application of a particular Furstenberg theorem
or of the correspondence principle; see
\cite{Furstenberg1967,Furstenberg1981}.

The paper is primarily expository, organized around the two exact
smoothing criteria, but the multiplicative semigroup action, the
description of the M\"obius orbit, and the accompanying
measure--category phenomenon appear to be new, or at least are not
commonly stated in this form; no claim is made that these observations
furnish a new method for proving RH.  \Cref{sec:modern-randomness}
places them beside developments that came much later than Denjoy's
heuristic --- random multiplicative functions, short-interval
cancellation, averaged forms of Chowla's conjecture, and recent
limiting-distribution results --- and complete proofs of the elementary
reformulations are included throughout so that the article may be read
independently.

\section{Exponential smoothing and the Mellin transform}

The starting point is the Mellin transform of the discrete Laplace sum.

\begin{proposition}\label{prop:Mellin}
For \(\Re s>1\),
\begin{equation}\label{eq:Mellin-identity}
 \Mellin\Phi(s)
 :=
 \int_0^\infty \Phi(t)t^{s-1}\,\dd t
 =
 \frac{\Gamma(s)}{\zeta(s)}.
\end{equation}
\end{proposition}

\begin{proof}
When \(\sigma=\Re s>1\), Tonelli's theorem and
\(|\mu(n)|\leq1\) give
\[
 \sum_{n=1}^{\infty}|\mu(n)|
 \int_0^\infty \e^{-nt}t^{\sigma-1}\,\dd t
 \leq
 \Gamma(\sigma)\sum_{n=1}^{\infty}\frac1{n^\sigma}
 <\infty.
\]
Hence the sum and integral may be interchanged:
\[
 \begin{aligned}
 \int_0^\infty \Phi(t)t^{s-1}\,\dd t
 &=
 \sum_{n=1}^{\infty}\mu(n)
 \int_0^\infty \e^{-nt}t^{s-1}\,\dd t\\
 &=
 \Gamma(s)\sum_{n=1}^{\infty}\frac{\mu(n)}{n^s}\\
 &=
 \frac{\Gamma(s)}{\zeta(s)}.
 \end{aligned}
\]
\end{proof}

A single \(L^p\) bound already yields a zero-free half-plane.

\begin{proposition}\label{prop:Lp-zero-free}
Let \(1<p<\infty\).  If
\[
 \Phi\in L^p(0,\infty),
\]
then \(\zeta(s)\neq0\) throughout the half-plane
\[
 \Re s>\frac1p.
\]
\end{proposition}

\begin{proof}
Let \(p'=p/(p-1)\).  At infinity we have the elementary estimate
\[
 |\Phi(t)|
 \leq\sum_{n=1}^{\infty}\e^{-nt}
 =\frac{\e^{-t}}{1-\e^{-t}},
 \qquad t\geq1,
\]
so the Mellin integral converges absolutely there for every \(s\in\C\).

Near zero, H\"older's inequality gives
\[
 \int_0^1|\Phi(t)|t^{\sigma-1}\,\dd t
 \leq
 \|\Phi\|_{L^p(0,1)}
 \left(
 \int_0^1 t^{(\sigma-1)p'}\,\dd t
 \right)^{1/p'}.
\]
The integral on the right is finite precisely when
\[
 (\sigma-1)p'>-1,
\]
or equivalently \(\sigma>1/p\).  Standard dominated-convergence
arguments show that \(\Mellin\Phi(s)\) is holomorphic in
\(\Re s>1/p\).\footnote{In more detail: the H\"older bound above is
locally uniform in \(\sigma\) on compact subsets of
\(\{\Re s>1/p\}\), so \(s\mapsto\int_0^1\Phi(t)t^{s-1}\,\dd t\) is a
locally uniform limit, as \(N\to\infty\), of the entire functions
\(s\mapsto\int_{1/N}^1\Phi(t)t^{s-1}\,\dd t\); by Morera's theorem
(applied to any closed triangle in \(\Re s>1/p\), together with
Fubini's theorem to justify interchanging the contour integral with
the \(t\)-integral) the limit function is itself holomorphic there.
Combined with the everywhere-convergent tail integral
\(\int_1^\infty\Phi(t)t^{s-1}\,\dd t\), this gives holomorphy of
\(\Mellin\Phi(s)\) on \(\Re s>1/p\).}

By \cref{prop:Mellin}, this holomorphic function agrees with
\(\Gamma(s)/\zeta(s)\) in the nonempty overlap \(\Re s>1\).
It therefore gives a holomorphic continuation of that quotient to
\(\Re s>1/p\).  Since \(\Gamma\) has no zeros, \(\zeta\) can have no
zero in this half-plane.\footnote{Explicitly: if \(\zeta(s_0)=0\) for
some \(s_0\) with \(\Re s_0>1/p\), then, since \(\Gamma(s_0)\neq0\),
the meromorphic function \(\Gamma(s)/\zeta(s)\) has a pole at
\(s_0\). But \(\Mellin\Phi\) is holomorphic, hence finite, at
\(s_0\), and the identity \(\Mellin\Phi(s)=\Gamma(s)/\zeta(s)\) holds
throughout the connected open set \(\{\Re s>1/p\}\setminus
\zeta^{-1}(0)\); letting \(s\to s_0\) along this set forces
\(\Mellin\Phi(s_0)=\infty\), a contradiction.}
\end{proof}

\section{The Laplace \texorpdfstring{\(L^p\)}{Lp} criterion}
\begin{theorem}\label{thm:laplace}
For \(\Phi\) defined by \eqref{eq:Phi-def}, the following are
equivalent:
\begin{enumerate}[label=\textup{(\roman*)}]
\item the Riemann hypothesis holds;
\item \(\Phi\in L^p(0,\infty)\) for every \(1\leq p<2\);
\item \(\Phi\in L^{2-\varepsilon}(0,\infty)\) for every
      \(0<\varepsilon\leq1\).
\end{enumerate}
\end{theorem}

\begin{proof}
The equivalence of \textup{(ii)} and \textup{(iii)} is only a change of
notation.

Assume RH.  By \eqref{eq:mertens-rh}, for every \(\delta>0\),
\[
 M(x)=O_\delta(x^{1/2+\delta}).
\]
Abel summation gives
\begin{equation}\label{eq:Abel-Laplace}
 \Phi(t)
 =
 t\int_1^\infty M(x)\e^{-tx}\,\dd x.
\end{equation}
Indeed, the boundary contribution at infinity vanishes, while the
lower endpoint is harmless because \(M(x)=0\) for \(0\leq x<1\).

For \(0<t\leq1\), \eqref{eq:Abel-Laplace} yields
\[
 \begin{aligned}
 |\Phi(t)|
 &\ll_\delta
 t\int_1^\infty x^{1/2+\delta}\e^{-tx}\,\dd x\\
 &\leq
 t\int_0^\infty x^{1/2+\delta}\e^{-tx}\,\dd x\\
 &=
 \Gamma\!\left(\frac32+\delta\right)
 t^{-1/2-\delta}.
 \end{aligned}
\]
Fix \(1\leq p<2\).  Choose \(\delta>0\) sufficiently small that
\[
 p\left(\frac12+\delta\right)<1.
\]
Then
\[
 \int_0^1|\Phi(t)|^p\,\dd t
 \ll_{\delta,p}
 \int_0^1t^{-p(1/2+\delta)}\,\dd t
 <\infty.
\]
As noted in the proof of \cref{prop:Lp-zero-free}, \(\Phi\) decays
exponentially at infinity.  Hence
\[
 \Phi\in L^p(0,\infty)
 \qquad(1\leq p<2).
\]

Conversely, assume \textup{(ii)}.  Given any \(\sigma>1/2\), choose
\(p<2\) sufficiently close to \(2\) that
\[
 \frac1p<\sigma.
\]
By \cref{prop:Lp-zero-free}, \(\zeta(s)\neq0\) for
\(\Re s>1/p\), and in particular at every point with
\(\Re s=\sigma\).  Thus \(\zeta\) has no zeros in
\(\Re s>1/2\).  The functional equation and the symmetry of the
nontrivial zeros about the critical line imply RH.
\end{proof}

\begin{remark}[The endpoint]\label{rem:endpoint}
The theorem deliberately uses \(p<2\).  The proof does not assert that
RH implies \(\Phi\in L^2(0,\infty)\), and no endpoint statement is
needed for the equivalence.  The exponent \(2\) is critical because the
Mertens bound on RH leads to the borderline scale
\(|\Phi(t)|\lesssim t^{-1/2-o(1)}\) as \(t\downarrow0\).
\end{remark}

\begin{remark}[One exponent gives a zero-free half-plane]
For a fixed \(p>1\), the condition
\[
 \Phi\in L^p(0,\infty)
\]
already implies the zero-free region
\[
 \zeta(s)\neq0
 \qquad\left(\Re s>\frac1p\right)
\]
by \cref{prop:Lp-zero-free}.  The full RH criterion arises by allowing
\(p\) to approach \(2\) from below.
\end{remark}

\section{Local Fourier smoothing}

For the Fourier formulation, one keeps the finite exponential sum and averages its size on a short arc.  The relevant criterion is as follows.

\begin{theorem}[Local Fourier-moment criterion]\label{thm:fourier}
Fix \(c>0\).  The following assertions are equivalent.
\begin{enumerate}[label=\textup{(\roman*)}]
\item The Riemann hypothesis holds.
\item For every \(\eta>0\) and every finite \(q\ge1\),
\[
\Mcal_{q,c}(N)=O_{\eta,q,c}(N^\eta).
\]
\item For every \(\eta>0\), there is an unbounded set
\(Q_\eta\subseteq[1,\infty)\) such that
\[
\Mcal_{q,c}(N)=O_{\eta,q,c}(N^\eta)
\qquad(q\in Q_\eta).
\]
\end{enumerate}
\end{theorem}

\begin{remark}\label{rem:iii-precise}
The set \(Q_\eta\) may depend on \(\eta\).  In the proof of
\textup{(iii)}\(\Rightarrow\)\textup{(i)}, one therefore fixes
\(\eta\) first and then chooses a sufficiently large
\(q\in Q_\eta\).
\end{remark}

\section{From Denjoy's random walk to modern M\texorpdfstring{\"o}{o}bius randomness}
\label{sec:modern-randomness}

\subsection{The independent-sign picture}

Let \((\varepsilon_n)_{n\geq1}\) be independent Rademacher variables,
so that
\[
 \Pp(\varepsilon_n=1)=\Pp(\varepsilon_n=-1)=\frac12.
\]
The random walk
\[
 R_N=\sum_{n\leq N}\varepsilon_n
\]
satisfies, for every \(\delta>0\),
\[
 R_N=O_\delta(N^{1/2+\delta})
 \qquad\text{almost surely}.
\]
More precisely, the law of the iterated logarithm gives
(see, for example, \cite{Kahane})
\[
 \limsup_{N\to\infty}
 \frac{R_N}{\sqrt{2N\log\log N}}=1
 \qquad\text{almost surely}.
\]
Replacing the deterministic signs of the M\"obius function by
independent signs therefore produces an almost-sure analogue of
\eqref{eq:mertens-rh}.  This is the simplest form of the Denjoy
heuristic.

A model retaining the squarefree support is
\[
 R_N^\mu=\sum_{n\leq N}\mu(n)^2\varepsilon_n.
\]
Its variance satisfies
\[
 \E|R_N^\mu|^2
 =\sum_{n\leq N}\mu(n)^2
 \sim\frac6{\pi^2}N.
\]
Similarly, one may consider the random trigonometric polynomials
\[
 Q_N(t,\omega)
 =
 \frac1{\sqrt N}
 \sum_{n\leq N}\mu(n)^2\varepsilon_n(\omega)\e^{2\pi i nt}.
\]
For each fixed \(q<\infty\), the Khintchine inequalities
\cite{Khinchin,Kahane} give uniform moment estimates for \(Q_N\).
These independent models are useful benchmarks, but they deliberately
remove the multiplicative correlations of the arithmetic function.

\subsection{Random multiplicative functions}

The modern probabilistic theory keeps multiplicativity.  Wintner
\cite{Wintner} introduced a Rademacher random multiplicative function
by choosing independent signs on the primes, extending
multiplicatively to squarefree integers, and setting the value equal to
zero on integers divisible by a square.  Thus randomness is placed at
the primes, not independently at every integer.  A Steinhaus random
multiplicative function is obtained instead by choosing independent
uniform random variables on the unit circle at the primes and extending
completely multiplicatively.  These models are closer in spirit to
\(\mu\), Dirichlet characters, and Euler products.

The distinction matters.  The terms in a random multiplicative
function are not independent.  Harper proved, for both the Steinhaus
and Rademacher models, that
\begin{equation}\label{eq:harper-first-moment}
 \E\left|\sum_{n\leq x} f(n)\right|
 \asymp
 \frac{\sqrt{x}}{(\log\log x)^{1/4}},
\end{equation}
so even the first absolute moment is smaller than the naive
independent-sign scale \(\sqrt{x}\) \cite{HarperLowMoments}.  The
appearance of this logarithmic factor is one indication that
multiplicativity changes the probabilistic picture in an essential
way.  Harper's work also connects these moments with critical
multiplicative chaos.  In later work he carried related ``better than
square-root'' estimates back to families of deterministic character and
zeta sums, including sums with bounded multiplicative weights such as
\(\mu(n)\) \cite{HarperTypical}.

There are now several forms of central limit behavior, but they depend
on the sum being considered.  Chatterjee and Soundararajan
\cite{ChatterjeeSoundararajan} proved a central limit theorem for
random multiplicative functions in suitable short intervals.
Soundararajan and Xu \cite{SoundararajanXu} later established central
limit theorems for a variety of subsets and weighted sums, while also
emphasizing that the unrestricted long partial sum does not have the
naive Gaussian behavior.  Gorodetsky and Wong
\cite{GorodetskyWongMartingale} obtained, for a wide class of weighted
random multiplicative sums, non-Gaussian limiting distributions with a
random variance arising from an associated random measure.  Their more
recent work \cite{GorodetskyWongLimit} determines the limiting
distribution of a normalized long Steinhaus partial sum and relates the
answer to critical multiplicative chaos.  Taken together, these results
show that even within the well-controlled setting of random
multiplicative functions, the limiting behavior of partial sums is
delicate: whether it is Gaussian, non-Gaussian, or governed by critical
multiplicative chaos depends on the specific sum, weight, and scale
under consideration.

A very recent result of Harper, Soundararajan, and Xu
\cite{HarperSoundararajanXu2026} gives a useful further contrast between
long and short sums.  For a Steinhaus random multiplicative function and
short intervals \([x,x+y]\) with \(y\to\infty\) and \(y=o(x)\), they
prove a Gaussian limiting distribution after the appropriate
normalization.  When \(y\asymp x\), no normalization gives a
non-degenerate Gaussian limit.  Thus the phrase ``behaves randomly''
has to be qualified: the answer depends both on the random model and on
the scale at which it is observed.

\subsection{What is known for the deterministic M\"obius function}

The probabilistic models should not obscure the progress made for
\(\mu\) itself.  Matom\"aki and Radziwi\l\l \cite{MatomakiRadziwill}
proved that M\"obius cancellation occurs in almost all intervals
\([x,x+\psi(x)]\) whenever \(\psi(x)\to\infty\), however slowly.  This
was striking in part because such short-interval cancellation had not
previously been available even under RH in comparable generality.
Their methods led, with Tao, to averaged forms of Chowla's conjecture
\cite{MatomakiRadziwillTao}.  Later work of Matom\"aki, Radziwi\l\l,
Tao, Ter\"av\"ainen, and Ziegler proved higher-order uniformity for
bounded non-pretentious multiplicative functions in short intervals on
average, with applications to the Liouville function and averaged
Chowla-type statements \cite{MRTTZ}.

These deterministic results and the theory of random multiplicative
functions address different questions, but they meet at a common
point: square-root cancellation is only the beginning of the story.
One must also understand correlations, short scales, limiting laws, and
the constraints imposed by multiplicativity.  Sarnak's conjecture is
another expression of this broader point of view, formulated in terms
of correlations with low-complexity deterministic dynamics.

\subsection{The local Fourier model used below}

The local Fourier construction in this paper keeps the coefficients
\(\mu(n)\) unchanged and randomizes only the evaluation point:
\[
 X_{N,c}=P_N(U_{N,c}).
\]
Thus
\[
 \E|X_{N,c}|^q=\Mcal_{q,c}(N)^q.
\]
The value controlling RH, \(P_N(0)=M(N)/\sqrt N\), occurs at a point
of probability zero.  The relevant elementary observation is that the
bounded degree of \(P_N\) forces a large value at zero to persist on a
short interval, where it can be detected by sufficiently high local
moments.  This is a different use of probability from the random
multiplicative models just discussed.

\section{A probabilistic form of Denjoy's heuristic}
\label{sec:probabilistic-form}

We now make precise the heuristic recalled in the introduction.  This
is \emph{not} Wintner's random multiplicative model of
\cref{sec:modern-randomness}: the coefficients below are independent
at every integer, so multiplicativity is intentionally discarded.  The
advantage is that Denjoy's elementary random-walk idea, described by
Edwards \cite[pp.~263--266]{Edwards}, becomes an exact statement in a
very simple probability space, in which the summatory function behaves
like a random walk and should display square-root cancellation.  The
Riemann hypothesis then appears as the deterministic arithmetic
analogue of an almost-sure random statement.

For the smoothing criterion, the same idea can be formulated exactly
by replacing the deterministic M\"obius coefficients with independent
random variables.
Let
\[
G=\{-1,0,1\},
\]
identified with the additive group \(\mathbb Z/3\mathbb Z\), and let
\[
\Omega=G^{\mathbb N}
\]
be endowed with its Haar probability measure \(\nu\).  For
\(\omega\in\Omega\), write
\[
\xi_n(\omega)=\omega(n).
\]
Then the random variables \((\xi_n)_{n\geq1}\) are independent and
uniformly distributed on \(\{-1,0,1\}\).  In particular,
\[
\mathbb E\xi_n=0,
\qquad
\mathbb E|\xi_n|^2=\frac23.
\]

For \(\omega\in\Omega\), define the random discrete Laplace transform
\[
\Phi_\omega(t)
=
\sum_{n=1}^{\infty}\xi_n(\omega)e^{-nt},
\qquad t>0.
\]
The series converges absolutely for every \(t>0\).

\begin{proposition}[Almost-sure \(L^p\)-integrability]
\label{prop:random-laplace}
For every \(0<p<2\),
\[
\nu\left\{
\omega\in\Omega:
\Phi_\omega\in L^p(0,\infty)
\right\}=1.
\]
Equivalently, for every \(0<\varepsilon<2\),
\[
\nu\left\{
\omega\in\Omega:
\Phi_\omega\in L^{2-\varepsilon}(0,\infty)
\right\}=1.
\]
\end{proposition}

\begin{proof}
For each \(t>0\), independence and the vanishing of the means give
\[
\begin{aligned}
\mathbb E|\Phi_\omega(t)|^2
&=
\mathbb E\left|
\sum_{n=1}^{\infty}\xi_n e^{-nt}
\right|^2\\
&=
\sum_{n=1}^{\infty}\mathbb E|\xi_n|^2e^{-2nt}\\
&=
\frac23\sum_{n=1}^{\infty}e^{-2nt}\\
&=
\frac{2}{3(e^{2t}-1)}.
\end{aligned}
\]
Hence, for \(0<t\leq1\),
\[
\mathbb E|\Phi_\omega(t)|^2\ll t^{-1}.
\]
If \(0<p<2\), Lyapunov's inequality yields
\[
\mathbb E|\Phi_\omega(t)|^p
\leq
\left(\mathbb E|\Phi_\omega(t)|^2\right)^{p/2}
\ll_p t^{-p/2}.
\]
Since \(p<2\),
\[
\int_0^1 t^{-p/2}\,dt<\infty.
\]
Tonelli's theorem therefore gives\footnote{The map
\((\omega,t)\mapsto|\Phi_\omega(t)|^p\) is jointly measurable: it is a
pointwise limit, as \(K\to\infty\), of the partial sums
\(\bigl|\sum_{n\leq K}\xi_n(\omega)e^{-nt}\bigr|^p\), each of which is
jointly continuous in \((\omega,t)\) since it depends on only finitely
many coordinates of \(\omega\).}
\[
\mathbb E\left[
\int_0^1|\Phi_\omega(t)|^p\,dt
\right]
=
\int_0^1\mathbb E|\Phi_\omega(t)|^p\,dt
<\infty.
\]
It follows that
\[
\int_0^1|\Phi_\omega(t)|^p\,dt<\infty
\]
for \(\nu\)-almost every \(\omega\).

At infinity no probabilistic estimate is required, since for every
\(\omega\in\Omega\),
\[
|\Phi_\omega(t)|
\leq
\sum_{n=1}^{\infty}e^{-nt}
=
\frac1{e^t-1}.
\]
Thus
\[
\int_1^\infty|\Phi_\omega(t)|^p\,dt<\infty
\]
for every \(\omega\), and the conclusion follows.
\end{proof}

The almost-sure statement has a useful topological complement.  Let
\[
A_p=
\left\{
\omega\in G^{\mathbb N}:
\Phi_\omega\in L^p(0,\infty)
\right\}.
\]

\begin{proposition}[Measure versus topology]
\label{prop:measure-topology}
The following assertions hold.
\begin{enumerate}[label=\textup{(\roman*)}]
\item If \(0<p<1\), then
\[
A_p=G^{\mathbb N}.
\]
\item If \(1\leq p<2\), then
\[
\nu(A_p)=1,
\qquad
\operatorname{int}(A_p)=\varnothing.
\]
\end{enumerate}
\end{proposition}

\begin{proof}
For every \(\omega\in G^{\mathbb N}\),
\[
|\Phi_\omega(t)|
\leq
\sum_{n=1}^{\infty}e^{-nt}
=
\frac1{e^t-1}.
\]
As \(t\downarrow0\), the right-hand side is asymptotic to \(1/t\).
Hence, for \(0<p<1\),
\[
\int_0^1|\Phi_\omega(t)|^p\,dt<\infty
\]
for every \(\omega\).  Exponential decay at infinity then gives
\(A_p=G^{\mathbb N}\).

Now let \(1\leq p<2\).  By
\cref{prop:random-laplace}, \(\nu(A_p)=1\).  To prove that \(A_p\)
has empty interior, consider a nonempty basic open set in the product
topology.  Such a set fixes only finitely many coordinates, say
\[
\omega_1,\ldots,\omega_N.
\]
Complete this finite sequence by setting
\[
\omega_n=1
\qquad(n>N).
\]
For the resulting \(\omega\),
\[
\Phi_\omega(t)
=
\sum_{n=1}^{N}\omega_n e^{-nt}
+
\frac{e^{-(N+1)t}}{1-e^{-t}}.
\]
The finite sum remains bounded as \(t\downarrow0\), whereas
\[
\frac{e^{-(N+1)t}}{1-e^{-t}}
\sim\frac1t.
\]
Therefore
\[
|\Phi_\omega(t)|\sim\frac1t
\qquad(t\downarrow0),
\]
and so
\[
\int_0^1|\Phi_\omega(t)|^p\,dt=\infty
\qquad(p\geq1).
\]
Thus every nonempty basic open set contains a point outside \(A_p\),
which proves
\[
\operatorname{int}(A_p)=\varnothing.
\]
\end{proof}

The proposition gives a precise smoothing form of Denjoy's old
probabilistic intuition.  Independent centered coefficients of bounded
variance satisfy the required \(L^{2-\varepsilon}\)-integrability
almost surely.  For the M\"obius sequence, the analogous statement is
not a consequence of independence; it is equivalent to the arithmetic
cancellation encoded by the Riemann hypothesis.

\section{The multiplicative semigroup action and the M\texorpdfstring{\"o}{o}bius orbit}

The random coefficient space carries a natural action of the
multiplicative semigroup
\[
S=\mathbb N_{+}.
\]
For \(\ell\in S\), define
\[
(\sigma_\ell\omega)(n)=\omega(\ell n),
\qquad
\omega\in G^{\mathbb N}.
\]
The corresponding action on discrete Laplace transforms is
\[
T_\ell\Phi_\omega
=
\Phi_{\sigma_\ell\omega},
\]
that is,
\[
(T_\ell\Phi_\omega)(t)
=
\sum_{n=1}^{\infty}\omega(\ell n)e^{-nt}.
\]
Since
\[
\sigma_\ell\circ\sigma_m=\sigma_{\ell m},
\]
the family \((\sigma_\ell)_{\ell\geq1}\) is a commutative semigroup
action.

This construction is close in spirit to Furstenberg's use of dynamical
systems in number theory \cite{Furstenberg1967,Furstenberg1981}.  A basic
principle of that viewpoint is to encode arithmetic operations as
commuting transformations and then study the resulting system by means
of invariant measures, recurrence, ergodicity, and orbit structure.  In
the present setting multiplication in \(\mathbb N_+\) acts directly on
the index set through \(n\mapsto \ell n\), hence on the coefficient
space through \(\sigma_\ell\).  This is not the same system as
Furstenberg's classical multiplication maps on the circle, such as
\(x\mapsto 2x\pmod 1\) and \(x\mapsto 3x\pmod 1\), and no use of the
correspondence principle is intended.  The analogy is rather that an
arithmetic structure is converted into dynamics.  Here the distinguished
arithmetic point is the M"obius sequence, and the Riemann hypothesis
becomes a membership statement for that point in a natural measurable
subset of the dynamical space.

\begin{proposition}
\label{prop:semigroup-action}
For every \(\ell\geq2\), the transformation
\[
\sigma_\ell:G^{\mathbb N}\longrightarrow G^{\mathbb N}
\]
preserves Haar measure.  It is noninvertible, exact, strongly mixing,
ergodic, and has infinite Kolmogorov--Sinai entropy.  Consequently, the
full \(\mathbb N_{+}\)-action is ergodic.
\end{proposition}

\begin{proof}
Let \(C\) be a cylinder set determined by distinct coordinates
\(n_1,\dots,n_r\):
\[
C=
\left\{
\omega:
\omega(n_j)=a_j,\ 1\leq j\leq r
\right\}.
\]
Then
\[
\sigma_\ell^{-1}C
=
\left\{
\omega:
\omega(\ell n_j)=a_j,\ 1\leq j\leq r
\right\}.
\]
The indices \(\ell n_1,\dots,\ell n_r\) are distinct, and therefore
\[
\nu(\sigma_\ell^{-1}C)=3^{-r}=\nu(C).
\]
Since cylinder sets generate the Borel sigma-algebra,
\(\sigma_\ell\) preserves \(\nu\).

The map is not invertible for \(\ell>1\), since the coordinates whose
indices are not divisible by \(\ell\) are discarded.  Every
\(n\in\mathbb N\) can be written uniquely as
\[
n=r\ell^k,
\qquad \ell\nmid r.
\]
Thus \(\mathbb N\) is the disjoint union of the chains
\[
\{r,r\ell,r\ell^2,\ldots\},
\qquad \ell\nmid r.
\]
On each chain, \(\sigma_\ell\) acts as the one-sided Bernoulli shift.
Hence \(\sigma_\ell\) is isomorphic to a countable product of
one-sided Bernoulli shifts.

This description implies strong mixing and ergodicity.\footnote{These
properties of the one-sided Bernoulli shift, and their persistence
under countable independent products, are classical; see, e.g.,
\cite{CFS} or \cite{Walters}.} It also gives
exactness: the sigma-algebra
\[
\sigma_\ell^{-k}\mathcal B
\]
is generated by coordinates with indices divisible by \(\ell^k\), and
the intersection
\[
\bigcap_{k\geq1}\sigma_\ell^{-k}\mathcal B
\]
is trivial modulo null sets.

Each chain contributes entropy \(\log 3\).  Since there are countably
many independent chains, one obtains
\[
h_\nu(\sigma_\ell)=+\infty.
\]
Finally, a set invariant under the full semigroup is in particular
invariant under \(\sigma_2\), and therefore has measure zero or one.
\end{proof}

The semigroup action is also mixing when the multiplier tends to
infinity.  If \(A\) and \(B\) are cylinder sets, then \(A\) depends on
coordinates in a finite set \(E\subset\mathbb N\), while
\(\sigma_\ell^{-1}B\) depends on coordinates in \(\ell F\) for another
finite set \(F\).  For all sufficiently large \(\ell\),
\[
E\cap \ell F=\varnothing,
\]
so independence gives
\[
\nu(A\cap\sigma_\ell^{-1}B)=\nu(A)\nu(B).
\]
Approximation by cylinder sets yields
\[
\nu(A\cap\sigma_\ell^{-1}B)
\longrightarrow
\nu(A)\nu(B)
\qquad(\ell\to\infty)
\]
for arbitrary measurable sets \(A\) and \(B\).

The almost-sure \(L^p\) property is stable along the entire
multiplicative orbit.  If \(1\leq p<2\), then
\[
\nu(A_p)=1,
\qquad
A_p=
\left\{
\omega:
\Phi_\omega\in L^p(0,\infty)
\right\}.
\]
Since every \(\sigma_\ell\) preserves \(\nu\),
\[
\nu(\sigma_\ell^{-1}A_p)=1.
\]
Because the semigroup is countable,
\[
\nu\left(
\bigcap_{\ell\geq1}\sigma_\ell^{-1}A_p
\right)=1.
\]
Thus, for almost every \(\omega\),
\[
\Phi_{\sigma_\ell\omega}\in L^p(0,\infty)
\qquad\text{for every }\ell\geq1.
\]

The Möbius sequence defines a highly nongeneric point of this action.
For \(\omega(n)=\mu(n)\),
\[
(T_\ell\Phi_\mu)(t)
=
\sum_{n=1}^{\infty}\mu(\ell n)e^{-nt}.
\]
If \(\ell\) is not squarefree, then \(\mu(\ell n)=0\) for every \(n\),
and consequently
\[
T_\ell\Phi_\mu=0.
\]
If \(\ell\) is squarefree, then
\[
\mu(\ell n)
=
\begin{cases}
\mu(\ell)\mu(n),&(n,\ell)=1,\\
0,&(n,\ell)>1,
\end{cases}
\]
so that
\[
(T_\ell\Phi_\mu)(t)
=
\mu(\ell)
\sum_{\substack{n\geq1\\(n,\ell)=1}}
\mu(n)e^{-nt}.
\]
The orbit of the Möbius point therefore reflects its multiplicative
arithmetic structure and differs sharply from a generic Bernoulli
orbit.

\section{Measure, category, and the Möbius point}
\label{sec:measure-category}

The preceding random model admits a natural formulation in descriptive
set theory.  For \(p>0\), let
\[
A_p=
\left\{
\omega\in G^{\mathbb N}:
\Phi_\omega\in L^p(0,\infty)
\right\}.
\]
The set \(A_p\) is Borel and, in fact, has low complexity in the Borel
hierarchy.

For \(m\geq1\), define
\[
I_m(\omega)
=
\int_{1/m}^{m}|\Phi_\omega(t)|^p\,dt.
\]
On the compact interval \([1/m,m]\), the series defining
\(\Phi_\omega\) converges uniformly, uniformly in \(\omega\).
Consequently, the map
\[
\omega\longmapsto I_m(\omega)
\]
is continuous.

Moreover,
\[
\int_0^\infty|\Phi_\omega(t)|^p\,dt
=
\sup_{m\geq1}I_m(\omega).
\]
It follows that
\[
A_p
=
\bigcup_{K=1}^{\infty}
\bigcap_{m=1}^{\infty}
\left\{
\omega:I_m(\omega)\leq K
\right\}.
\]
Thus \(A_p\) is an \(F_\sigma\) subset of \(G^{\mathbb N}\).

\begin{proposition}[Measure and category]
\label{prop:measure-category}
For \(0<p<1\),
\[
A_p=G^{\mathbb N}.
\]
For \(1\leq p<2\),
\[
\nu(A_p)=1,
\]
but \(A_p\) is meagre.  Equivalently,
\[
G^{\mathbb N}\setminus A_p
\]
has measure zero and contains a dense \(G_\delta\) set.
\end{proposition}

\begin{proof}
The first statement was proved in
\cref{prop:measure-topology}.  Suppose \(1\leq p<2\).  The full-measure
statement follows from \cref{prop:random-laplace}.

Write
\[
A_p=\bigcup_{K=1}^{\infty}C_K,
\qquad
C_K=
\bigcap_{m=1}^{\infty}
\left\{
\omega:I_m(\omega)\leq K
\right\}.
\]
Each \(C_K\) is closed.  We show that it has empty interior.

Let \(U\) be a nonempty basic open cylinder.  It fixes only finitely
many coordinates, say \(\omega_1,\ldots,\omega_N\).  Extend this finite
sequence by setting
\[
\omega_n=1
\qquad(n>N).
\]
For the resulting point,
\[
\Phi_\omega(t)
=
\sum_{n=1}^{N}\omega_n e^{-nt}
+
\frac{e^{-(N+1)t}}{1-e^{-t}},
\]
and therefore
\[
|\Phi_\omega(t)|\sim\frac1t
\qquad(t\downarrow0).
\]
Hence
\[
\int_0^1|\Phi_\omega(t)|^p\,dt=\infty
\qquad(p\geq1),
\]
so this point does not belong to \(A_p\), and in particular does not
belong to \(C_K\).  Thus no nonempty cylinder is contained in \(C_K\).
Each \(C_K\) is therefore nowhere dense, and \(A_p\), being their
countable union, is meagre.
\end{proof}

The proposition exhibits a sharp difference between probabilistic and
topological genericity.  For \(1\leq p<2\), a sequence chosen according
to Haar measure belongs to \(A_p\) almost surely, whereas a Baire-generic
sequence does not.

To formulate the simultaneous condition, let
\[
A_{<2}
=
\bigcap_{\substack{p\in\mathbb Q\\1\leq p<2}}A_p.
\]
Then
\[
\nu(A_{<2})=1,
\]
while \(A_{<2}\) is meagre.

\medskip
\noindent\textbf{Borel-hierarchy terminology.}
Recall that \(\boldsymbol{\Sigma}^0_1\) denotes the class of open sets
and \(\boldsymbol{\Pi}^0_1\) the class of closed sets.  The class
\(\boldsymbol{\Sigma}^0_2\) consists of countable unions of closed
sets (the \(F_\sigma\) sets), and \(\boldsymbol{\Pi}^0_3\) consists
of countable intersections of \(\boldsymbol{\Sigma}^0_2\) sets.
Since each \(A_p\) is an \(F_\sigma\) set and
\[
A_{<2}
=
\bigcap_{\substack{p\in\mathbb Q\\1\leq p<2}}A_p,
\]
it follows that \(A_{<2}\in\boldsymbol{\Pi}^0_3\).  Equivalently,
\(A_{<2}\) is a Borel set of class at most
\(\boldsymbol{\Pi}^0_3\).  The words ``at most'' indicate only an
upper bound on its Borel complexity; no optimality claim is made.  See
\cite[Chapter~II, especially \S11]{Kechris} for the Borel hierarchy.\footnote{Meagreness of \(A_{<2}\) follows because
\(A_{<2}\subseteq A_1\) (taking \(p=1\) in the intersection),
\(A_1\) is meagre by \cref{prop:measure-category}, and every subset of
a meagre set is meagre.  Moreover, the countable, rational-indexed
intersection \(A_{<2}\) coincides with the continuum-indexed condition
``\(\Phi_\omega\in L^p(0,\infty)\) for every real \(p\in[1,2)\)'' of
\cref{thm:laplace}(ii): if \(f\in L^{p_1}\cap L^{p_2}\) for rationals
\(p_1<p_2\) in \([1,2)\), splitting the domain of integration into
\(\{|f|\leq1\}\) and \(\{|f|>1\}\) shows \(f\in L^r\) for every real
\(r\in[p_1,p_2]\); hence membership in \(A_p\) for a dense set of
rationals approaching \(2\) already forces membership in \(A_p\) for
every real \(p\in[1,2)\).}

The Laplace criterion for the Möbius sequence can therefore be written
as
\[
\mathrm{RH}
\quad\Longleftrightarrow\quad
\mu\in A_{<2}.
\]
This is a membership problem for one distinguished arithmetic point in
a natural Borel set that is large in measure and small in category.
The descriptive-set-theoretic complexity of the set is not the source
of the difficulty: the arithmetic structure of the Möbius point is.

Thus it would be misleading to describe RH as ``Borel undecidable.''
Borel complexity concerns the definability of subsets of a Polish
space, whereas formal undecidability concerns provability in a fixed
axiomatic system.  The appropriate conclusion is that the smoothing
criterion places RH at the intersection of three notions of
genericity:
\[
\begin{array}{c|c}
\text{Haar measure} & A_{<2}\text{ is typical},\\
\text{Baire category} & A_{<2}\text{ is exceptional},\\
\text{arithmetic} & \mu\in A_{<2}\text{ is equivalent to RH}.
\end{array}
\]
This table is the precise counterpart of Denjoy's heuristic remark,
recalled in the introduction, that the Riemann hypothesis holds ``with
probability one.''  The M\"obius sequence \(\mu\) is a single,
non-random point of \(G^{\mathbb N}\); nothing in the argument assigns
a literal probability to the fixed proposition RH.  What the measure
row of the table does say, rigorously, is that if the M\"obius signs
were replaced by a Haar-random element of the same coefficient space,
the resulting Laplace transform would lie in \(A_{<2}\) almost surely
--- so that the property equivalent to RH is exactly the property
possessed by \emph{almost every} point of the space in which \(\mu\)
sits.  In that precise, comparative sense --- and only in that sense
--- Denjoy's old remark is made rigorous: RH is precisely the assertion that
the M\"obius sequence has a property which is measure-theoretically
typical in the ambient coefficient space, even though that same property
is topologically exceptional.

\section{A local moment-to-point-value estimate}
\label{sec:local-fourier}

The key estimate converts a local \(L^q\)-norm into control of the value at the center of the arc.

\begin{lemma}\label{lem:persistence}
Let
\[
 S_N(t)=\sum_{n=1}^N a_n\e^{2\pi i nt},
 \qquad |a_n|\leq1,
\]
and let \(A=|S_N(0)|\).  Then
\[
 |S_N'(t)|\leq \pi N(N+1)
\]
for every \(t\).  Consequently,
\[
 |S_N(t)|\geq\frac A2
 \qquad\text{whenever}\qquad
 |t|\leq \frac{A}{2\pi N(N+1)}.
\]
\end{lemma}

\begin{proof}
Differentiation gives
\[
 S_N'(t)=2\pi i\sum_{n=1}^N n a_n\e^{2\pi i nt},
\]
and hence
\[
 |S_N'(t)|
 \leq2\pi\sum_{n=1}^N n
 =\pi N(N+1).
\]
The second assertion follows from the mean value estimate
\[
 |S_N(t)-S_N(0)|
 \leq \pi N(N+1)|t|,
\]
which gives \(|S_N(t)|\geq A-\pi N(N+1)|t|\), and the right-hand side
is \(\geq A/2\) precisely when \(|t|\leq A/(2\pi N(N+1))\).  (When
\(A=0\) the stated conclusion holds trivially, the hypothesis
\(|t|\leq0\) forcing \(t=0\).)
\end{proof}

For \(r>0\), define the normalized local \(L^q\)-mean
\begin{equation}\label{eq:Bq-def}
 B_q(S_N;r)
 =
 \left(
 \frac1{2r}\int_{-r}^{r}|S_N(t)|^q\,\dd t
 \right)^{1/q}.
\end{equation}

\begin{proposition}[Local moment-to-point-value inequality]
\label{prop:moment-point}
For every \(q\geq1\), \(N\geq1\), and \(r>0\),
\begin{equation}\label{eq:moment-point}
 |S_N(0)|
 \leq
 \max\left\{
 2B_q(S_N;r),
 \,
 C_q\bigl(N^2r\bigr)^{1/(q+1)}
 B_q(S_N;r)^{q/(q+1)}
 \right\},
\end{equation}
where
\[
 C_q:=\bigl(2^{q+2}\pi\bigr)^{1/(q+1)}
\]
is a constant depending only on \(q\) (not on \(N\) or \(r\)).
\end{proposition}

\begin{proof}
Write \(A=|S_N(0)|\) and \(L=\pi N(N+1)\).

Suppose first that \(A/(2L)\geq r\).  By
\cref{lem:persistence}, \(|S_N(t)|\geq A/2\) throughout
\([-r,r]\).  Therefore
\[
 B_q(S_N;r)\geq\frac A2,
\]
which gives the first term in \eqref{eq:moment-point}.

Suppose now that \(A/(2L)<r\).  Again by
\cref{lem:persistence},
\[
 |S_N(t)|\geq\frac A2
 \qquad\text{for }|t|\leq\frac{A}{2L}.
\]
The interval on the right has length \(A/L\), and therefore
\[
 B_q(S_N;r)^q
 \geq
 \frac1{2r}\frac{A}{L}\left(\frac A2\right)^q.
\]
It follows that
\[
 A^{q+1}
 \leq 2^{q+1}Lr\,B_q(S_N;r)^q.
\]
Now write \(L=\pi N(N+1)=\pi N^2(1+N^{-1})\), so that
\[
 A^{q+1}
 \leq
 \bigl(2^{q+1}\pi(1+N^{-1})\bigr)\,(N^2r)\,B_q(S_N;r)^q,
\]
and hence
\[
 A
 \leq
 \bigl(2^{q+1}\pi(1+N^{-1})\bigr)^{1/(q+1)}
 (N^2r)^{1/(q+1)}
 B_q(S_N;r)^{q/(q+1)}.
\]
Since \(N\geq1\) gives \(1+N^{-1}\leq2\), the prefactor
\(\bigl(2^{q+1}\pi(1+N^{-1})\bigr)^{1/(q+1)}\) is bounded by
\(\bigl(2^{q+2}\pi\bigr)^{1/(q+1)}=C_q\), a quantity that no longer
depends on \(N\).  Replacing the (possibly smaller) \(N\)-dependent
prefactor by the larger constant \(C_q\) only weakens the inequality,
so it remains valid, proving the second term of
\eqref{eq:moment-point} with the constant \(C_q\) fixed in the
statement of the proposition.
\end{proof}

Applying the proposition to the Möbius polynomial
\[
 S_N(t)=\sum_{n\leq N}\mu(n)\e^{2\pi i nt},
 \qquad\text{so that}\qquad
 P_N(t)=\frac1{\sqrt N}S_N(t),
\]
gives the quantitative estimate that will imply the converse direction
of \cref{thm:fourier}.  We first record the elementary identity linking
the two normalizations \(B_q\) and \(\Mcal_{q,c}\): comparing
definitions \eqref{eq:Bq-def} and \eqref{eq:localmoment} at \(r=c/N\),
\[
 B_q(S_N;c/N)^q
 =
 \frac{N}{2c}\int_{-c/N}^{c/N}|S_N(t)|^q\,\dd t
 =
 \frac{N}{2c}\int_{-c/N}^{c/N}\bigl(\sqrt N|P_N(t)|\bigr)^q\,\dd t
 =
 N^{q/2}\,\Mcal_{q,c}(N)^q,
\]
that is,
\begin{equation}\label{eq:BqMq}
 B_q(S_N;c/N)=\sqrt N\,\Mcal_{q,c}(N).
\end{equation}

\begin{corollary}\label{cor:Mertens-local}
For every \(q\geq1\), \(c>0\), and \(N\geq1\),
\begin{equation}\label{eq:Mertens-local}
 |M(N)|
 \leq C_{q,c}
 \max\left\{
 \sqrt N\,\Mcal_{q,c}(N),
 \,
 N^{\frac12+\frac{1}{2(q+1)}}
 \Mcal_{q,c}(N)^{\frac{q}{q+1}}
 \right\}.
\end{equation}
\end{corollary}

\begin{proof}
Take \(r=c/N\) in \cref{prop:moment-point}, applied to the Möbius
polynomial \(S_N\), and use identity \eqref{eq:BqMq}.  The first term
of \eqref{eq:moment-point} becomes \(2\sqrt N\,\Mcal_{q,c}(N)\), while
the second becomes
\[
 C_q\,(N^2c/N)^{1/(q+1)}
 \bigl(\sqrt N\,\Mcal_{q,c}(N)\bigr)^{q/(q+1)}
 =
 C_q\,c^{1/(q+1)}
 N^{\frac12+\frac{1}{2(q+1)}}
 \Mcal_{q,c}(N)^{q/(q+1)}.
\]
Since \(S_N(0)=M(N)\), taking
\(C_{q,c}:=\max\{2,\,C_q\,c^{1/(q+1)}\}\) --- which majorizes both the
factor \(2\) in the first term and the factor \(C_q\,c^{1/(q+1)}\) in
the second --- gives \eqref{eq:Mertens-local}.
\end{proof}

\begin{remark}
For one fixed exponent \(q\), a subpolynomial bound for
\(\Mcal_{q,c}(N)\) yields
\[
 M(N)=O_\varepsilon\left(
 N^{1/2+1/(2(q+1))+\varepsilon}
 \right).
\]
The loss \(1/(2(q+1))\) tends to zero as \(q\to\infty\).
This explains why arbitrarily high finite moments, rather than one
fixed moment, recover RH.
\end{remark}

\section{Proof of the local Fourier criterion}
\label{sec:fourier-proof}

The theorem follows by applying this estimate to the normalized Möbius polynomial.

\begin{proof}[Proof of \cref{thm:fourier}]
Assume first that RH holds.  Then for every \(\delta>0\),
\[
 M(x)=O_\delta(x^{1/2+\delta}).
\]
Abel summation (partial summation) gives, for \(t\in\mathbb R\), the
exact identity
\begin{equation}\label{eq:Abel}
 S_N(t)
 =
 M(N)\e^{2\pi iNt}
 -
 2\pi it\int_1^N M(x)\e^{2\pi ixt}\,\dd x.
\end{equation}
Indeed, with \(f(x)=\e^{2\pi ixt}\) and the convention \(M(x)=0\) for
\(x<1\), the general partial summation formula
\[
 \sum_{y<n\leq N}\mu(n)f(n)
 =
 M(N)f(N)-M(y)f(y)-\int_y^N M(x)f'(x)\,\dd x
\]
applied with any \(y\in[0,1)\) gives \(M(y)=0\) and
\(\sum_{y<n\leq N}\mu(n)f(n)=\sum_{n\leq N}\mu(n)f(n)=S_N(t)\), so it
reduces exactly to \eqref{eq:Abel} with no residual boundary term.
Consequently,
\[
 |S_N(t)|
 \ll_\delta
 N^{1/2+\delta}
 +
 |t|\int_1^N x^{1/2+\delta}\,\dd x
 \ll_\delta
 N^{1/2+\delta}(1+N|t|).
\]
For \(|t|\leq c/N\), this gives
\[
 |P_N(t)|\ll_{\delta,c}N^\delta.
\]
Therefore
\[
 \Mcal_{q,c}(N)\ll_{\delta,q,c}N^\delta
\]
for every finite \(q\geq1\).  This proves
\textup{(i)}\(\Rightarrow\)\textup{(ii)}, and
\textup{(ii)}\(\Rightarrow\)\textup{(iii)} is immediate.

Assume now \textup{(iii)}, in the precise form of \cref{rem:iii-precise}:
for every \(\eta>0\) there is an unbounded set \(Q_\eta\subseteq[1,\infty)\)
such that \(\Mcal_{q,c}(N)=O_{\eta,q,c}(N^\eta)\) for every
\(q\in Q_\eta\).

Let \(\varepsilon>0\).  The two choices below must be made in this
order, since the set of admissible exponents \(Q_\eta\) depends on
\(\eta\).  First choose \(\eta>0\) so small that
\[
 \eta<\frac{\varepsilon}{3}.
\]
By hypothesis \(Q_\eta\) is unbounded, so we may then choose
\(q\in Q_\eta\) large enough that
\[
 \frac1{2(q+1)}<\frac{\varepsilon}{3}.
\]
Since \(q\in Q_\eta\), the hypothesis gives
\[
 \Mcal_{q,c}(N)=O_{\eta,q,c}(N^\eta).
\]
The first term in \eqref{eq:Mertens-local} is
\[
 O(N^{1/2+\eta}),
\]
while the second is
\[
 O\left(
 N^{1/2+1/(2(q+1))+\eta q/(q+1)}
 \right)
 =
 O(N^{1/2+\varepsilon}).
\]
Thus
\[
 M(N)=O_\varepsilon(N^{1/2+\varepsilon}).
\]
Since this holds for every \(\varepsilon>0\), the classical Mertens
criterion implies RH.
\end{proof}

\section{The critical scale \texorpdfstring{\(N^{-1}\)}{N to the minus one}}
The scale \(N^{-1}\) is distinguished for two related reasons.
The general background on trigonometric polynomials and their local
behavior may be found in Zygmund \cite{Zygmund}.

First, it is the smallest period among the Fourier modes occurring in \(P_N\).
For the individual mode
\[
  \e^{2\pi i n t},
\]
the period in the variable \(t\) is \(1/n\), because
\[
  \e^{2\pi i n(t+1/n)}=\e^{2\pi i nt}.
\]
Thus the highest-frequency mode, corresponding to \(n=N\), has the
smallest period \(1/N\).  Consequently, when \(t\) varies by an amount
comparable with \(1/N\), the highest-frequency terms can change by an
amount of order one.  In particular, if \(t=c/N\) and \(n\) is comparable
with \(N\), then the phase \(2\pi nt\) is of order one.  Hence the random
variable \(P_N(U_{N,c})\) samples a genuinely varying trigonometric
polynomial rather than an almost constant perturbation of \(P_N(0)\).
By contrast, on intervals satisfying \(|t|\ll 1/N\), all the phases
\(2\pi nt\), \(n\leq N\), are small and the polynomial varies very little.

Second, RH itself controls the polynomial on this scale.  The estimate
deduced from \eqref{eq:Abel} is
\[
 |P_N(t)|\ll_\varepsilon N^\varepsilon(1+N|t|).
\]
Hence RH immediately gives subpolynomial control on arcs satisfying
\[
 r_N\leq N^{-1+o(1)}.
\]

This yields a modest generalization.

\begin{corollary}\label{cor:variable-scale}
Let \(r_N>0\) satisfy
\[
 Nr_N=N^{o(1)}.
\]
If RH holds, then, for every finite \(q\geq1\) and every
\(\varepsilon>0\),
\[
 \left(
 \frac1{2r_N}\int_{-r_N}^{r_N}|P_N(t)|^q\,\dd t
 \right)^{1/q}
 =O_{\varepsilon,q}(N^\varepsilon).
\]
\end{corollary}

\begin{proof}
The estimate
\[
 |S_N(t)|\ll_\delta N^{1/2+\delta}(1+N|t|)
\]
established in the proof of \cref{thm:fourier} (via \eqref{eq:Abel}
and the Mertens bound under RH) holds for \emph{every} real \(t\), not
merely for \(|t|\leq c/N\).  Dividing by \(\sqrt N\) gives
\[
 |P_N(t)|\ll_\delta N^\delta(1+N|t|)
 \qquad(t\in\mathbb R),
\]
so that, for \(|t|\leq r_N\),
\[
 |P_N(t)|\ll_\delta N^\delta(1+Nr_N).
\]
Raising to the \(q\)-th power and averaging over \((-r_N,r_N)\) gives
\[
 \left(
 \frac1{2r_N}\int_{-r_N}^{r_N}|P_N(t)|^q\,\dd t
 \right)^{1/q}
 \ll_{\delta,q}
 N^\delta(1+Nr_N).
\]
By hypothesis \(Nr_N=N^{o(1)}\), so the right-hand side is
\(N^{\delta+o(1)}\).  Given \(\varepsilon>0\), choose
\(\delta<\varepsilon/2\); then for all sufficiently large \(N\) the
right-hand side is \(\leq N^{\varepsilon}\), which proves the claim.
\end{proof}

\begin{remark}[On the maximal shrinking scale]
It is natural to ask how large the shrinking radius \(r_N\) may be while
subpolynomial local moment bounds still follow from RH alone.  The
calculation above shows that
\[
 r_N=N^{-1+o(1)}
\]
is the largest scale obtained directly from the classical RH estimate
for \(M(x)\).  This does not prove that larger arcs cannot satisfy such
estimates.  Bounds on larger arcs would require additional cancellation
in additively twisted Möbius sums and may therefore be strictly stronger
than RH.  Uniform estimates for exponential sums with M\"obius coefficients
go back to Davenport \cite{Davenport}.  Subsequent refinements and
investigations of their maximal size include the work of Hajela and
Smith \cite{HajelaSmith}, Baker and Harman \cite{BakerHarman}, and
Maier and Sankaranarayanan \cite{MaierSankaranarayanan}.
\end{remark}

\begin{question}\label{q:max-scale}
Determine the largest sequence \(r_N\to0\) for which RH alone implies
\[
 \left(
 \frac1{2r_N}\int_{-r_N}^{r_N}|P_N(t)|^q\,\dd t
 \right)^{1/q}
 =N^{o(1)}
\]
for every finite \(q\).  Is \(r_N=N^{-1+o(1)}\) optimal under RH
alone?
\end{question}

\section{Tail and Orlicz formulations}
The moment criterion admits equivalent-looking probabilistic
variations.  Let
\[
 X_{N,c}=P_N(U_{N,c}).
\]
If, for arbitrarily large \(q\),
\[
 \Pp(|X_{N,c}|>\lambda)
 \leq C_{q,\eta,c}N^{q\eta}\lambda^{-q},
 \qquad \lambda>0,
\]
together with a mild truncation or integrability condition ensuring
the corresponding \(q\)-th moment bound, then
\cref{thm:fourier} implies RH.

A stronger Kahane-type condition is uniform subgaussianity modulo a
subpolynomial factor:
\begin{equation}\label{eq:subgaussian}
 \E\exp\left(
 \frac{|X_{N,c}|^2}{C_\eta N^{2\eta}}
 \right)
 \leq C_\eta
 \qquad\text{for every }\eta>0.
\end{equation}
Indeed, \eqref{eq:subgaussian} implies
\[
 \|X_{N,c}\|_{L^q}
 \ll_\eta \sqrt q\,N^\eta
\]
for every finite \(q\), and therefore implies RH by
\cref{thm:fourier}.  Condition \eqref{eq:subgaussian} is stronger than
the moment criterion and should be viewed as a deterministic Möbius
analogue of the behavior of random-sign Fourier polynomials.

\section{Flatness and related polynomial questions}
\label{sec:flatness}

A trigonometric polynomial
\[
 Q_N(t)=\sum_{n=0}^{N}a_n\e^{2\pi int}
\]
is called \emph{flat} when its modulus is nearly constant on the
circle after a suitable normalization.  For example, if
\(\|Q_N\|_2=1\), a sequence \((Q_N)\) is uniformly flat when
\[
 \sup_{t\in\T}\bigl||Q_N(t)|-1\bigr|\longrightarrow0.
\]
One may also speak of \(L^p\)-flatness when
\[
 \bigl\||Q_N|-1\bigr\|_{L^p(\T)}\longrightarrow0.
\]

A weaker notion is \emph{\(L^p\)-semiflatness}.  This means only that
the \(L^p\)-norms remain uniformly bounded:
\[
 \sup_N\|Q_N\|_{L^p(\T)}<\infty.
\]
Thus flatness says that the modulus becomes almost constant, whereas
semiflatness merely prevents the polynomials from developing
excessively large peaks in the \(L^p\)-average sense.

For the normalized Möbius polynomials
\[
 P_N(t)=\frac1{\sqrt N}\sum_{n\leq N}\mu(n)\e^{2\pi int},
\]
one has
\[
 \|P_N\|_2^2
 =
 \frac1N\sum_{n\leq N}\mu(n)^2
 \longrightarrow\frac6{\pi^2}.
\]
Hence their global \(L^2\)-size is automatically controlled.  The
nontrivial questions concern \(L^q\)-norms with \(q>2\), or the local
moments considered in this paper.

\label{sec:erdos-flat}

A particularly important class consists of the \emph{Littlewood
polynomials}
\[
 L_N(z)=\sum_{n=0}^{N}a_nz^n,
 \qquad a_n\in\{-1,1\}.
\]
By Parseval's identity,
\[
 \frac1{2\pi}\int_0^{2\pi}
 |L_N(\e^{it})|^2\,\dd t=N+1,
\]
so their natural size on the unit circle is \(\sqrt{N+1}\).

Erd\H{o}s asked whether one can choose the signs \(a_n\) so that the
modulus of \(L_N\) remains comparable with this natural size at every
point of the unit circle.  More precisely, the problem is whether
there exist absolute constants \(0<\delta<\Delta<\infty\) such that,
for every sufficiently large \(N\), one can find a Littlewood
polynomial satisfying
\begin{equation}\label{eq:erdos-littlewood-flat}
 \delta\sqrt{N+1}
 \leq |L_N(z)|
 \leq \Delta\sqrt{N+1},
 \qquad |z|=1.
\end{equation}
This problem was recorded by Erd\H{o}s in his 1957 collection of
unsolved problems \cite{Erdos1957}.  Littlewood later conjectured that
such two-sided flat Littlewood polynomials do exist.

The problem was solved affirmatively by Balister, Bollob\'as, Morris,
Sahasrabudhe, and Tiba \cite{BalisterEtAl}: absolute constants
\(\delta,\Delta>0\) satisfying \eqref{eq:erdos-littlewood-flat} exist.
Thus Littlewood polynomials can be uniformly flat up to fixed
multiplicative constants.

This should be distinguished from \emph{ultraflatness}.  A sequence
of \(L^2\)-normalized polynomials is ultraflat if
\[
 \sup_{|z|=1}
 \left|
 \frac{|L_N(z)|}{\sqrt{N+1}}-1
 \right|
 \longrightarrow0.
\]
The theorem of Balister et al.\ gives bounded distortion between two
fixed positive constants, but not convergence of the normalized
modulus to \(1\).  For polynomials with arbitrary unimodular complex
coefficients, Kahane constructed ultraflat sequences; the coefficients
in such constructions need not be restricted to \(\{-1,1\}\).

The Möbius polynomials
\[
 P_N(t)=\frac1{\sqrt N}
 \sum_{n\leq N}\mu(n)\e^{2\pi int}
\]
are not Littlewood polynomials because \(\mu(n)\) may vanish.  They
nevertheless belong to the same broad family of polynomials with
highly restricted coefficients.  The question studied in this note
is also different from Erd\H{o}s--Littlewood flatness: we do not ask
for a uniform two-sided estimate on the whole circle.  We ask instead
for subpolynomial control of high moments on a shrinking arc centered
at the distinguished point \(t=0\).  The common theme is the extent
to which arithmetic or combinatorial coefficients can prevent large
peaks and deep troughs of a Fourier polynomial.

For normalized Möbius polynomials on the whole circle, one has
\[
 \|P_N\|_2^2
 =
 \frac1N\sum_{n\leq N}\mu(n)^2
 \longrightarrow\frac6{\pi^2}.
\]
Consequently, global \(L^q\)-bounds for \(q<2\) are automatic.  Global
bounds for \(q>2\) are much more restrictive.  Work on
\(L^p\)-semiflatness shows that sufficiently strong bounds for
arbitrarily large exponents imply RH; see el Abdalaoui
\cite{Abdalaoui}.

The local criterion differs in two respects.  It uses normalized
Lebesgue measure on a shrinking arc rather than Haar measure on the
whole circle, and it yields a direct quantitative estimate for the
Mertens function through \cref{eq:Mertens-local}.  The critical arc
contains the distinguished point \(t=0\), but that point itself has
probability zero.  The recovery mechanism is the persistence of a
large point value imposed by the bandwidth of the polynomial.

\section{Comparison of the two smoothing criteria}

The two principal equivalences can be written compactly as
\[
\mathrm{RH}
\quad\Longleftrightarrow\quad
\Phi\in\bigcap_{1\le p<2}L^p(0,\infty),
\]
and, for each fixed \(c>0\),
\[
\mathrm{RH}
\quad\Longleftrightarrow\quad
\Mcal_{q,c}(N)=N^{o(1)}
\quad\text{for arbitrarily large finite }q.
\]

The first condition is global on \((0,\infty)\), but its arithmetic
content is concentrated at the endpoint \(t=0\).  The second is local
on the circle and centers the averaging interval at the point
\(t=0\), where the Fourier polynomial equals \(M(N)/\sqrt N\).
In both cases the threshold is governed by square-root cancellation.

Thus the two criteria are different analytic expressions of the same
arithmetic phenomenon.  In the independent model, the Laplace
\(L^{2-\varepsilon}\) condition is almost sure but, for
\(1\le p<2\), not topologically robust: the corresponding set has
empty interior in \(G^{\mathbb N}\).  For the M\"obius sequence the
same integrability condition becomes equivalent to RH.  Smoothing makes the cancellation
visible in a function-space norm, but does not by itself produce it.

\section{Historical and expository context}

For a broad account of equivalent formulations of RH, Broughan's
volumes \cite{BroughanI,BroughanII,BroughanIII} are especially useful.
Volume~II has a strong analytic orientation and includes, among many
other topics, the Nyman--Beurling circle of ideas, Weil's explicit
formula, discrete measures, and horocycle criteria.  Chapter~10 treats
the author's 1994 paper \cite{VerjovskyDiscrete} in detail and places
it next to the criterion of Zagier and its dynamical extensions.  This
background is close to the point of view of the present paper because
Mellin inversion is central in both settings.

The Laplace criterion belongs to the classical family of real-variable
and transform-theoretic reformulations of RH.  B\'aez-Duarte
\cite{BaezDuarte} relates zero-free half-planes to \(L^p\)-properties
of the transformed Mertens function \(M(1/t)\), in the setting of
Nyman--Beurling-type criteria.  The present formulation instead uses
the exponentially smoothed sum \(\Phi(t)\).

The same discrete Laplace transform has recently been studied by
Liflandsky \cite{Liflandsky}.  That work gives, under an assumption on
the simplicity of the zeros, an explicit formula involving the zeros
of \(\zeta\), and discusses pointwise estimates as criteria for RH.
The argument in this article is independent of such an explicit
formula and uses only Abel summation, H\"older's inequality, and the
Mellin identity \eqref{eq:Mellin-identity}.

A related, and considerably deeper, illustration of what RH buys for
moments is Harper's sharp conditional bound
\(\int_T^{2T}|\zeta(1/2+it)|^{2k}\,dt\ll_k T\log^{k^2}T\) for every
fixed \(k\geq0\), obtained assuming RH by working directly with the
moments rather than passing through bounds for large values
\cite{HarperConditionalMoments}.  This is a related moment-theoretic manifestation of the same underlying
question --- how much RH controls the size of \(\zeta\) --- and is
considerably more precise than anything the present elementary methods can
reach; it is recorded here as a pointer toward the depth of the moment
problem on the critical line itself, which lies beyond the scope of the
smoothing criteria developed below.

We have not found the exact equivalence of \cref{thm:fourier} stated in
this form in the sources consulted.  Its proof is short and follows
standard principles, so we make no priority claim for the equivalence
itself; it may well be folklore.  Recording it here serves both an
expository and a structural purpose: it gives a transparent bridge
among M\"obius cancellation, local Fourier smoothing, high moments, and
the classical Mertens criterion.  The dynamical semigroup formulation
and the measure--category comparison developed elsewhere in the paper
provide additional viewpoints that are not part of this folklore
observation.

\begin{remark}[Relation to the author's earlier note]
The local Fourier-moment criterion of \cref{thm:fourier}, together with
the supporting persistence and moment-to-point-value estimates of
\cref{sec:local-fourier}, first appeared in a more narrowly focused
note by the author, \emph{Local Moments of M\"obius Fourier Polynomials
and the Riemann Hypothesis} (arXiv:2607.25002).  The present article
places that criterion alongside the Laplace \(L^p\)-criterion of
\cref{thm:laplace}, the probabilistic model of
\cref{sec:probabilistic-form}, the multiplicative semigroup action, and
the measure--category comparison within a single expository framework.
These additional strands go beyond the scope of the earlier note.
\end{remark}

\section{Concluding remarks}

The Mertens criterion admits at least two simple smoothing
interpretations.  Exponential smoothing leads, through the Mellin
transform, to an \(L^p\)-integrability threshold at \(p=2\).  Local
Fourier smoothing leads to high-moment bounds on arcs whose natural
scale is \(1/N\).  Both formulations are elementary once the classical
criterion \eqref{eq:mertens-rh} is available.

Together they give two complementary analytic pictures of M\"obius
cancellation: Mellin integrability on the half-line and local moments of
trigonometric polynomials on the circle.  A third picture is dynamical.
The multiplicative semigroup \(\mathbb N_+\) acts on the coefficient
space by restriction to multiples, the M\"obius sequence is a
 distinguished arithmetic point of this action, and RH is equivalent to
membership of that point in the natural set
\[
A_{<2}=\bigcap_{1\leq p<2}
\{\omega:\Phi_\omega\in L^p(0,\infty)\}.
\]
The fact that this set has full Haar measure while being meagre gives a
concrete meeting point of probability, topology, arithmetic, and
dynamics.  This semigroup formulation is different from the dynamical
systems connection furnished by Sarnak's conjecture.  At the same time,
it belongs naturally to the broader Furstenberg philosophy of translating
arithmetic structure into dynamics: multiplication of indices becomes a
commuting semigroup action, and RH is expressed as a property of one
distinguished arithmetic point of that action.

The survey material of \cref{sec:modern-randomness} also shows why the
old phrase ``M\"obius behaves like a random sequence of signs'' should
be used with care.  Independent signs explain the square-root scale,
but multiplicativity creates correlations that lead to subtler moment
estimates and limiting laws; for the deterministic M\"obius and
Liouville functions, modern short-interval and uniformity results give a
different and increasingly precise meaning to pseudorandomness.

Thus the article has two purposes.  It surveys a line of ideas from
Denjoy and Wintner to recent work on random multiplicative functions,
and it records several exact reformulations and dynamical observations
that arise naturally from the same question.  These observations are
modest, but they are part of the mathematical content of the paper and
not merely devices of exposition.

Questions beyond the results proved here include the endpoint \(p=2\),
the maximal shrinking scale for local Fourier moments, and comparisons
with pointwise Laplace estimates, global semiflatness, random
multiplicative models, and other dynamical notions of M\"obius
randomness.

\section*{Acknowledgments}

The author would like to acknowledge Proyecto PAPIIT IN103324
(DGAPA, UNAM, M\'exico) for its financial support.  The author also
acknowledges the use of ChatGPT (OpenAI) and Claude (Anthropic) for
proofreading, bibliographic searches, and refinement of the exposition.
The mathematical responsibility for the manuscript remains entirely
with the author.

\bigskip
\noindent
\textsc{Alberto Verjovsky}\\
Instituto de Matem\'aticas\\
Universidad Nacional Aut\'onoma de M\'exico\\
Apartado Postal 273, Administraci\'on de Correos No.~3\\
C.P.~62251 Cuernavaca, Morelos, M\'exico\\
\texttt{albertoverjovsky@gmail.com}

\end{document}